\documentclass[12pt]{amsart}
\usepackage[top=30truemm,bottom=30truemm,left=25truemm,right=25truemm]{geometry}
\usepackage{mathrsfs}

\usepackage{color}
\usepackage{bm}
\usepackage{amsfonts,amssymb}
\usepackage{dsfont}
\usepackage{amscd}
\usepackage{extarrows}
\usepackage{amsmath}
\usepackage{mathrsfs}
\usepackage{enumerate}
\usepackage{amscd}
\usepackage[all]{xy}
\usepackage[pagebackref,colorlinks]{hyperref}

\theoremstyle{plain} 
\newtheorem{theorem}{\indent\bf Theorem}[section]

\theoremstyle{definition} 
\newtheorem{definition}[theorem]{\indent\bf Definition}

\newtheorem{problem}[theorem]{\indent\bf Problem}

\newtheorem{thm}{Theorem}[section]

\newtheorem{lem}[thm]{Lemma}

\theoremstyle{definition}

\theoremstyle{remark}
\newtheorem{rem}{Remark}[section]

\newcommand{\be}{\begin{equation}}
	\newcommand{\ee}{\end{equation}}
\newcommand{\bea}{\begin{eqnarray}}
	\newcommand{\eea}{\end{eqnarray}}
\newcommand{\ben}{\begin{eqnarray*}}
	\newcommand{\een}{\end{eqnarray*}}
\newcommand{\bt}{\begin{split}}
	\newcommand{\et}{\end{split}}
\newcommand{\bet}{\begin{equation}}

\begin{document}
		\title[On several problems in p-Bergman theory]{On several problems in p-Bergman theory}
		
		%
	\author[Y. Li]{Yinji Li}
	\address{Yinji Li:  Institute of Mathematics\\Academy of Mathematics and Systems Sciences\\Chinese Academy of
		Sciences\\Beijing\\100190\\P. R. China}
	\email{1141287853@qq.com}

		\begin{abstract}
In this paper, we first answer Chen-Zhang's  problem on $p$-Bergman metric proposed in \cite{CZ22}. Second, we   prove the  off-diagonal p-Bergman kernel function $K_p(z,w)$ is H\"older continuous of order (1-$\varepsilon$) about the second component when $p\textgreater1$ for any $\varepsilon>0$, which improves the corresponding  result of Chen-Zhang. Moreover, we prove the asymptotic behavior of the maximizer of  $p$-Bergman kernel as $p\rightarrow 1^-$. Finally, we give a characterization of  a class of    holomorphic functions  on $\mathbb{B}^1$ to be $L^p$-integrable.

		\end{abstract}
		
		\thanks{}

		\maketitle
		\tableofcontents
	\section{Introduction}

The $L^2$ Bergman theory was established by Stefan Bergman in the 1920s,  is one of the fundamental theories in several complex variables and complex geometry. The $L^2$ Bergman space on a  domain in $\mathbb C^n$ is the space of $L^2$ holomorphic functions on that domain, which can be easily shown to be a Hilbert space using the theory of normal families. The $L^2$ Bergman kernel as  the integral kernel of the evaluation functional on the $L^2$ Bergman space,  shares good properties such as real analyticity and reproducing property. The $L^2$ Bergman kernel function is obtained by evaluating the kernel on the diagonal. On a bounded domain in $\mathbb C^n$, the $L^2$ Bergman kernel function is  smooth and strictly plurisubharmonic, non-vanishing and thus induces an invariant  K\"ahler metric on that domain, which is known as  the $L^2$  Bergman metric.  The $L^2$ Bergman metric plays an important role in the study of bounded domains. The $L^2$ Bergman theory can be extended to the framework of Hermitian holomorphic vector bundles over complex manifolds, and  has important applications in the study of various important problems in  complex geometry and algebraic geometry. 

In comparison with the $L^2$ Bergman theory, the $L^p$ Bergman theory has not been well studied. In \cite{NZZ16}, Ning-Zhang-Zhou  initiate a systematic study of  the $L^p$ Bergman theory, and get a deep result that a bounded domain is pseudoconvex if and only if the $L^p$ Bergman kernel is exhaustive for some $p\in (0,2)$.   Recently,  Deng-Wang-Zhang-Zhou \cite{DWZZ20} proved  the following fundamental result that two bounded hyperconvex domains in $\mathbb C^n$ are biholomorphically equivalent if and only if the normed $L^p$ Bergman spaces associated to them are linearly isometric for some $p\in (0,2)$. This shows that the $L^p$ Bergman space is a complete biholomorphic linear isometric invariant of bounded hyperconvex domains in $\mathbb C^n$. However, it is well-known that the $L^2$ Bergman space can not determine the complex structure of bounded hyperconvex domains, say the punctured disc for example. Thus the result by Deng-Wang-Zhang-Zhou indicates that the $L^p$ Bergman space is a very important research object and the $L^p$ Bergman theory deserves further development. However, unlike the $L^2$ Bergman theory, the $L^p$ spaces are generally not Hilbert spaces, which poses essential difficulties for research. A basic problem such as computing the $L^p$ Bergman kernel is highly challenging, and even the $L^p$ Bergman kernel on the punctured disk in the complex plane cannot be computed so far. Therefore, new methods and tools need to be developed.

For a bounded domain $\Omega\subset\mathbb{C}^n$, we define $A^p(\Omega)$ to be the p-Bergman space of $L^p$ holomorphic functions on $\Omega$(throughout this paper the integrals are with respect to Lebesgue measure). As introduced in \cite{NZZ16}, the $p$-Bergman kernel  $K_p(z)$ is defined as 
$$K_p(z)=\sup_{f\in A^p(\Omega)\setminus \{0\}}\frac{|f(z)|^p}{\|f\|_p^p},$$
where $\|f\|_p=(\int_\Omega|f|^p)^{1/p}$. The $p$-Bergman kernel can also be defined via a minimizing problem which was first introduced by Bergman himself in the case $p=2$:
\begin{align*}
m_p(z):=\inf\{||f||_p:f\in A^p(\Omega),f(z)=1\}.
\end{align*}
By a normal family argument, we know that there exists at least one minimizer for $p\textgreater0$ and exactly one minimizer $m_p(\cdot,z)$ for $p\geq1$. 
 It is easy to see that  $K_p(z):=m_p(z)^{-p}$ for $p\textgreater0$. The the off-diagonal $p$-Berman kernal is defined as  $K_p(z,w):=m_p(z,w)K_p(w)$ for $p\geq1$.
 
  Recently, Chen-Zhang \cite{CZ22} explored further fundamental aspects of  the $L^p$ Bergman theory using variational methods.   They derived  reproducing formula for $L^p$ Bergman kernels and  show that the off-diagonal $L^p$ Bergman kernel ($p$-Bergman kernel for short) $K_p(z,\cdot)$ is H\"older continuous of order $\frac{1}{2}$ for $p>1$ and of order $\frac{1}{2(n+2)}$ for $p=1$. They also defined the $p$-Bergman mertic $B_p(z;X)$ and showed that the  $p$-Bergman metric $B_p(z;X) $ tends to the Carath\'eodory metric $C(z;X)$ as $p\rightarrow \infty$  and the generalized Levi form $i\partial\bar\partial\log K_p(z;X)$  is no less than $B_p(z;X)^2$ for $p\geq 2$ and $C(z;X)^2$ for $p\leq 2$. 
  
  Since it is well-known that $i\partial\bar\partial\log K_p(z;X)=B_p(z;X)^2$ for $p=2$, Chen-Zhang raised the following 

  \begin{problem}[{\cite[Problem 8]{CZ22}}]\label{Prob} Is it possible to conclude that $i\partial\bar\partial\log K_p(z;X)=B_p(z;X)^2$ for $2<p<+\infty$?
  \end{problem}
  
In this paper, we first answer Problem \ref{Prob} by establishing the following 
\begin{thm}\label{thm: main 1}
Let $\Omega$ be complete circular and bounded homogeneous domain in $\mathbb{C}^n$, we have for $X\neq0$,
\begin{align*}
i\partial\bar{\partial} K_p(0;X)\textgreater B_p(0;X)^2,\ p\textgreater2,
\end{align*}
\begin{align*}
i\partial\bar{\partial} K_p(0;X)\textless B_p(0;X)^2,\ p\textless2.
\end{align*}
\end{thm}

Second, by introducing a new iteration technique, we are able to improve the regularity of the off-diagonal $p$-Bergman kernels, namely we improve the order of the H\"older continuity from $\frac{1}{2}$ to $1-\varepsilon$ for any $\varepsilon>0$ and $p>1$.

\begin{thm}\label{thm:main 2}
Let $p\textgreater1$, $\varepsilon\textgreater0,\ S\subseteq\subseteq \Omega$, there exists $C=C(\varepsilon,S)$ such that for $z',z,w\in S$
\begin{align*}
|m_p(z',z)-m_p(z',w)|\leq C|z-w|^{1-\varepsilon}.
\end{align*}
Moreover, the off-diagonal $p$-Bergman kernel $K_p(z,\cdot)$ is  H\"older continuous of order $1-\varepsilon$.
\end{thm}


It is proved in \cite[Proposition 2.4, Proposition 2.5]{CZ22} that for $p\geq1$ the maximizer $f$ of $K_p(z)$ is unique under the condition $f(z)=1$. Actually, it is precisely $m_p(\cdot,z)$. But the uniqueness of the maximizer of $K_p(z)$ for $0<p<1$  is not known. We study the asymptotic behavior of the maximizers of $K_p(z)$ as $p\rightarrow 1^-$ and get the following

\begin{thm}\label{thm: main 3}
Let $p\textless1$, we define the metric $d(f,g):=\int_{\Omega}|f-g|^p$ on $A^p(\Omega)$. Denote $d_p(z):=\sup\{ d(f_p,g_p)\}$, where $\sup$ is taken over all pairs of maximizers $f_p,g_p$ of $K_p(z)$ satisfying $f_p(z)=g_p(z)=1$. Then, it holds that $$\forall z\in{\Omega}, \lim_{p \to 1^-}d_p(z)=0.$$
\end{thm}

Finally, we study $L^p$ Bergman space $A^p(\mathbb B ^1)$ on the unit disk $\mathbb B^1$. A charcterization  for a class of holomorphic functions on $\mathbb B^1$ to be $L^p$-integrable is established as follows.

\begin{thm}\label{thm: main 4}
Let $p\textgreater0$, there exists $C=C(p,A)$ such that, if $f\in \mathcal O(\mathbb{B}^1)$, $f(z)=\sum_{k=1}^{\infty}a_{\lambda_k}z^{\lambda_k}$ for some lacunary sequence $\{\lambda_k\}$ with constant $A$, 
\begin{align*}
C(p,A)^{-1}\int_0^1(\sum_{k=1}^{\infty}|a_{\lambda_k}r^{2\lambda_k}|)^{\frac{p}{2}}dr\leq \int_{\mathbb{B}^1}|f|^p\leq C(p,A)\int_0^1(\sum_{k=1}^{\infty}|a_{\lambda_k}r^{2\lambda_k}|)^{\frac{p}{2}}dr.
\end{align*}
In particular, a holomorphic function $f(z)=\sum_{k=1}^{\infty}a_{\lambda_k}z^{\lambda_k}$ for some lacunary sequence $\{\lambda_k\}$ with constant $A$ is $L^p$-integrable if and only if the integration  $\int_0^1(\sum_{k=1}^{\infty}|a_{\lambda_k}r^{2\lambda_k}|)^{\frac{p}{2}}dr$ is finite.

\end{thm}

\begin{rem}
Theorem \ref{thm: main 4} can also be used to give a similar characterization of a class of  holomorphic functions on the punctured disk to be $L^p$-integrable by considering the Laurent expansions.
\end{rem}

The structure of this paper is organized as follows. In \S \ref{sect: open prob}, we answer the open problem raised by Chen-Zhang, and prove Theorem \ref{thm: main 1}. In \S \ref{sect: regularity }, we prove the off-diagonal $p$-Bergman kernel is H\"older continuous of order $1-\varepsilon$, i.e.  Theorem \ref{thm:main 2} . In \S \ref{sect: asymptotic}, we study the asymptotic behavior of the maximizer of the $p$-Bergman kernel as $p\rightarrow 1^-$, i.e. Theorem \ref{thm: main 3}. Finally, in \S \ref{sect: char}, we give a characterization of a class of holomorphic functions on the unit disk to be $L^p$-integrable, i.e. Theorem \ref{thm: main 4}.

$\mathbf{Acknowlegements.}$
The author would like to express his sincere gratitude to Professor Zhiwei Wang and Professor Xiangyu Zhou for their guidence and encouragements. This research is supported by National Key R\&D Program of China (No. 2021YFA1002600).

\section{Chen-Zhang's problem}\label{sect: open prob}

In this section, we  answer the  Problem \ref{Prob} raised by Chen-Zhang.

\begin{definition}
A domain $\Omega\subseteq\mathbb{C}^n$ is said to be complete circular and bounded  homogeneous, if $\forall z\in\mathbb{C}^n,t\in\mathbb{C},|t|\leq1$, we have $tz\in\Omega.$
\end{definition}

We restate Theroem \ref{thm: main 1} as follows.
\begin{thm}
Let $\Omega$ be complete circular and bounded homogeneous domain in $\mathbb{C}^n$, we have for $X\neq0$,
\begin{align*}
i\partial\bar{\partial} K_p(0;X)\textgreater B_p(0;X)^2,\ p\textgreater2,
\end{align*}
\begin{align*}
i\partial\bar{\partial} K_p(0;X)\textless B_p(0;X)^2,\ p\textless2.
\end{align*}
\end{thm}
\begin{proof}
It follows from \cite[Theorem 2.3,Remark 2.1]{NZZ16} that, on $\Omega$,  we have $K_p(\cdot)=K_2(\cdot),\forall p\textgreater0$. In particular, $K_p(0)=K_2(0)=\frac{1}{\mathrm{vol}(\Omega)}$. It is clear that
\begin{align*}
i\partial\bar{\partial}\log K_p(z;X)=i\partial\bar{\partial}\log K_2(z;X).
\end{align*}
In the following, we  prove that
\begin{align*}
&B_p(z;X)\textless B_2(z;X), \ p\textgreater2\\
&B_p(z;X)\textless B_2(z;X), \ p\textless2.
\end{align*}
Recall the definition, $B_p(z;X):=K_p(z)^{-\frac{1}{p}}\cdot\sup_{f\in A^p(\Omega),f(z)=0,||f||_p\textgreater0}\frac{|Xf(z)|}{||f||_p}$. By a normal family argument we know that there exists maximizer of $B_p(0;X)$ and denote it by $f_p$. It follows from H\"older inequality that
\begin{align*}
||f_p||_p^2\cdot||1||_p^{p-2}\geq||f_p||_2^2,\ p\textgreater2.
\end{align*}
However the equality can not be achieved since $f_p\not\equiv 1$. Thus we get that
\begin{align*}
B_p(0;X)^2&=K_p(0)^{-\frac{2}{p}}\cdot\frac{|Xf_p(0)|^2}{||f_p||_p^2}\\
&\textless K_2(0)^{-1}\cdot\frac{|Xf_p(0)|^2}{||f_p||_2^2}\\
&\leq B_2(0;X).
\end{align*}
The case that $p\textless2$ can be proved by the same method.
\end{proof}

\section{H\"older continuity of $m_p(z,\cdot)$}\label{sect: regularity }
In this section, we prove the off-diagonal $p$-Bergman kernel is H\"older continuous of order $1-\varepsilon$
for any $\varepsilon>0$. More precisely, we prove the following
\begin{thm}\label{thm:holder con}
Let $p\textgreater1$, $\varepsilon\textgreater0,\ S\subseteq\subseteq \Omega$, there exists $C=C(\varepsilon,S)$ such that for $z',z,w\in S$
\begin{align*}
|m_p(z',z)-m_p(z',w)|\leq C|z-w|^{1-\varepsilon}.
\end{align*}
\end{thm}
Let us introduce an important function as follows
\begin{align*}
H_p(z,w):=K_p(z)+K_p(w)-\mathrm{Re}\{K_p(z,w)+K_p(w,z)\}.
\end{align*}

\subsection{The case of $1\textless p\leq2$}
In this section, we assume $1\textless p\leq2$ and prove Theorem \ref{thm:holder con}.
\begin{proof}
It follows from the proof of \cite[Lemma 4.5]{CZ22} that 
\begin{align*}
\int_{\Omega}|m_p(\cdot,z)-m_p(\cdot,w)|^p\leq \frac{C_p}{K_p(z)K_p(w)}[K_p(z)+K_p(w)]^{1-\frac{p}{2}}H_p(z,w)^{\frac{p}{2}}.
\end{align*}
This leads to $||m_p(\cdot,z)-m_p(\cdot,w)||_p\leq C(p,S)H_p(z,w)^{\frac{1}{2}}$. Next, we are going to establish an  estimate for $H_p(z,w)$.
\begin{align*}
|\frac{H_p(z,w)}{z-w}|
&=\frac{\mathrm{Re}\{K_p(z)[m_p(w,w)-m_p(w,z)]+K_p(w)[m_p(z,w)-m_p(z,z)]\}}{|z-w|}\\
&\leq K_p(z)\frac{|[m_p(w,w)-m_p(w,z)]-[m_p(z,w)-m_p(z,z)]|}{|z-w|} \\
&+|m_p(z,z)-m_p(z,w)|\frac{|K_p(z)-K_p(w)|}{|z-w|}.
\end{align*}
Since $K_p(\cdot)$ is locally Lipschitz by \cite[Proposition 2.11]{CZ22}, we know that $\frac{|K_p(z)-K_p(w)|}{|z-w|}\leq C(S)$. It follows from the sub mean-value property of plurisubharmonic function that $|m_p(z,z)-m_p(z,w)|\leq C||m_p(\cdot,z)-m_p(\cdot,w)||_p$, for some $C=C(S)$. In view of Cauchy integral formula, we know that $m_p(\cdot,z)-m_p(\cdot,w)$'s derivative is controlled by its $L^1$ norm, therefore we get 
\begin{align*}
\frac{|[m_p(w,w)-m_p(w,z)]-[m_p(z,w)-m_p(z,z)]|}{|z-w|}&\leq C||m_p(\cdot,z)-m_p(\cdot,w)||_1\\
&\leq ||m_p(\cdot,z)-m_p(\cdot,w)||_p
\end{align*}
All the facts above imply that, for some $C=C(S)$
\begin{align*}
H_p(z,w)\leq C||m_p(\cdot,z)-m_p(\cdot,w)||_p\cdot|z-w|.
\end{align*}
Combine this result with the fact that
\begin{align*}
||m_p(\cdot,z)-m_p(\cdot,w)||_p\leq C(p,S)H_p(z,w)^{\frac{1}{2}},
\end{align*}
for any number $\delta$ less than $\frac{1}{2}+\frac{1}{4}+\frac{1}{8}+...=1$, we get that
\begin{center}
$H_p(z,w)=o(|z-w|^{1+\delta})$,
\end{center}
\begin{center}
$||m_p(\cdot,z)-m_p(\cdot,w)||_p=o(|z-w|^{\delta})$.
\end{center}
The desired result follows from $|m_p(z',z)-m_p(z',w)|\leq C||m_p(\cdot,z)-m_p(\cdot,w)||_p\leq C|z-w|^{\delta}$.
\end{proof}
\subsection{The case of $p\textgreater2$}
In this section, we assume $p\textgreater2$ and prove Theorem \ref{thm:holder con}.
\begin{proof}
It follows from the proof of \cite[Theorem 4.7]{CZ22} that there exists an open set $U$ with $S\subseteq U\subseteq\subseteq\Omega$, and a constant $\alpha=\alpha(p,S,U)$, $C=C(p,S,U)$ such that
\begin{align*}
\int_U|m_p(\cdot,z)-m_p(\cdot,w)|^{\alpha}\leq C H_p(z,w)^{\frac{\alpha}{2}}.
\end{align*}
This leads to $||m_p(\cdot,z)-m_p(\cdot,w)||_{L^{\alpha}(U)}\leq CH_p(z,w)^{\frac{1}{2}}$. The rest part of proof is similar with the case $1\textless p\leq2$. We get that $\forall \delta\textless1$, there exists $C=C(\delta,S,U)$ such that
\begin{center}
$||m_p(\cdot,z)-m_p(\cdot,w)||_{L^{\alpha}(U)}\leq |z-w|^{\delta}$,
\end{center}
\begin{center}
$H_p(z,w)\leq C||m_p(\cdot,z)-m_p(\cdot,w)||_{L^{\alpha}(U)}|z-w|$.
\end{center}
The desired result follows from $|m_p(z',z)-m_p(z',w)|\leq C||m_p(\cdot,z)-m_p(\cdot,w)||_{L^{\alpha}(U)}\leq C|z-w|^{\delta}$.
\end{proof}

\section{Asymptotic Behavior of Maximizers of $K_p(z)$ as $p \rightarrow 1^{-}$} \label{sect: asymptotic}
We know that when $p\geq1$, the maximizer $f$ of $K_p(z)$ is unique under the condition $f(z)=1$. Actually, it is precisely the minimizer $m_p(\cdot,z)$ of $m_p(z)$. However, the uniqueness of the maximizer is not known for  $p\textless1$. Nevertheless, we can prove the following: 
\begin{thm}
Let $p\textless1$, $A^p(\Omega)$ is a metric space. We define the metric $d(f,g):=\int_{\Omega}|f-g|^p$, and the function $d_p(z):=\sup\{ d(f_p,g_p)\}$, where $\sup$ is taken over all pairs of maximizers $f_p,g_p$ of $K_p(z)$ satisfying $f_p(z)=g_p(z)=1$. Then, it holds that $$\forall z\in{\Omega}, \lim_{p \to 1}d_p(z)=0.$$
\end{thm}
\begin{proof}
We have $K_p(z)^{-1}=\frac{1}{\int_{\Omega}|f_p|^p}\leq \int_{\Omega}1=|\Omega|$. Therefore, we know that $\forall p_0\textless1,\{f_p\}_{p_0\textless p\textless 1}$ is a normal family. Thus, there exists a subsequence $\{f_{p_n}\}$ that converges uniformly on compact subsets of $\Omega$ to some $f$. For any $p_0 \textless s\textless1$, by Fatou's lemma, H\"older's inequality and \cite[Proposition 6.1(1)]{CZ22}, we get that
\begin{align*}
\int|f|^s&\leq \liminf_{n \to\infty}\int|f_{p_n}|^s \\
&\leq \lim_{n \to \infty}(\int|f_{p_n}|^{p_n})^{\frac{s}{p_n}}|\Omega|^{1-\frac{s}{p_n}}\\ 
&=\lim_{n \to \infty} K_{p_n}(z)^{-\frac{s}{p_n}}|\Omega|^{1-\frac{s}{p_n}}\\
&=K_1(z)^{-s}|\Omega|^{1-s}.
\end{align*}
It follows that $\int|f|=\lim\limits_{s \to 1}\int|f|^{s}\leq \lim\limits_{s \to1} K_1(z)^{-s}|\Omega|^{1-s}=K_1(z)^{-1}.$ Howerver, $f(z)=1$ implies that $f$ is a maximizer of $K_1(z)$ at $z$.

Next, we prove $\lim\limits_{n \to \infty} \int_{\Omega}|f_{p_n}-f|^{p_n}=0.$

For any $\varepsilon \textgreater 0$, there exists $U\subset\subset \Omega$, such that $\int_{U}|f|\textgreater K_1(z)^{-1}-\varepsilon.$ This means $\int_{\Omega-U}|f|\textless \varepsilon$. Moreover, for sufficiently large $n$, since $f_{p_n}$ uniformly converge to $f$ on any compact subset of $\Omega$, we know that $\int_{U}|f_{p_n}-f|^{p_n}\textless \varepsilon.$ On the other hand, by $|f_{p_n}-f|^{p_n}\leq|f_{p_n}|^{p_n}+|f|^{p_n}$, we can see that 
\begin{align*}
\int_{\Omega-U}|f_{p_n}-f|^{p_n} &\leq \int_{\Omega-U}(|f_{p_n}|^{p_n}+|f|^{p_n})\\
&\leq K_{p_n}(z)^{-1}-\int_{U}|f_{p_n}|^{p_n}+(\int_{\Omega-U}|f|)^{p_n}|\Omega|^{1-p_n}\\
&\leq K_{p_n}(z)^{-1}-(\int_{U}|f|^{p_n}-\varepsilon)+\varepsilon^{p_n}|\Omega|^{1-p_n}.
\end{align*}
Notice that $\lim\limits_{n \to \infty}\int_{U}|f|^{p_n}=\int_{U}|f|\textgreater K_1(z)^{-1}-\varepsilon$ and $\lim\limits_{n \to \infty} K_{p_n}(z)=K_1(z)$ (\cite[Proposition 6.1(1)]{CZ22}). Therefore, we can conclude that ${\limsup_{n \to \infty}}\int_{\Omega-U}|f_{p_n}-f|^{p_n}\leq 3\varepsilon.$ Since $\varepsilon$ is arbitrary, it follows that $\lim_{n \to \infty}\int_{\Omega}|f_{p_n}-f|^{p_n}=0.$

Below, we prove the theorem by contradiction. If there exists $\delta\textgreater0$ such that there exists a sequence $\{p_n\} $ converges to $1$, and $\int_{\Omega}|f_{p_n}-g_{p_n}|^{p_n}=d(f_{p_n},g_{p_n})\textgreater \delta$. By taking subsequences twice, we may assume that $f_{p_n}$ and $g_{p_n}$ both converge to the maximizer $m_1(\cdot,z)$ of $K_1(z)$, as described above. However, this leads to $\int_{\Omega}|f_{p_n}-g_{p_n}|^{p_n}\leq \int_{\Omega}|f_{p_n}-m_1(\cdot,z)|^{p_n}+\int_{\Omega}|g_{p_n}-m_1(\cdot,z)|^{p_n}\rightarrow0$, which is a contradiction.
\end{proof}

\section{Characterization of $L^p$-integrability of a class of holomorphic functions on $\mathbb{B}^1$ }\label{sect: char}
Let  $\Omega=\mathbb{B}^1=\{z\in\mathbb{C}:|z|\textless1\}$.  In this section, we give a characterization of holomorphic functions $f\in \mathcal O(\mathbb B^1)$ to be $L^p$-integrable.

\begin{definition}
A sequence $\{\lambda_k\}_{k\in\mathbb{N}^*}$ is called lacunary with constant $A$ if there exists $A\textgreater1$, such that $\lambda_{k+1}\geq A\lambda_k$.
\end{definition}
The main theorem of this section is following
\begin{thm}\label{thm: char}
Let $p\textgreater0$, there exists $C=C(p,A)$ such that, if $f\in \mathcal{O}(\mathbb{B}^1)$, $f(z)=\sum_{k=1}^{\infty}a_{\lambda_k}z^{\lambda_k}$ for some lacunary sequence $\{\lambda_k\}$ with constant $A\textgreater1$, 
\begin{align*}
C(p,A)^{-1}\int_0^1(\sum_{k=1}^{\infty}|a_{\lambda_k}r^{2\lambda_k}|)^{\frac{p}{2}}dr\leq \int_{\mathbb{B}^1}|f|^p\leq C(p,A)\int_0^1(\sum_{k=1}^{\infty}|a_{\lambda_k}r^{2\lambda_k}|)^{\frac{p}{2}}dr.
\end{align*}

\end{thm}

We need the following lemma(\cite[Theorem 3.6.4]{Gra08})
\begin{lem}\label{lem:fourier}
Let $T=[0,1]$, $1\leq\lambda_1\textless\lambda_2\textless...$ be a lacunary sequence with constant $A\textgreater1$. Set $\Gamma=\{\lambda_k:k\in\mathbb{N}^*\}$. Then for all $1\geq p\textless\infty$, there exists a constant $C_p(A)$ such that for all $f\in L^1(T)$, with $f(k)=0$ when $k\in \mathbb{N}^*- \Gamma$, we have
$$||f||_{L^p(T)}\leq C_p(A)||f||_{L^1(T)}.$$
Moreover, the converse inequality is also valid, hence all $L^p$ norms of lacunary Fourier sequence are equivalent for $1\leq p\textless\infty$.
\end{lem}

\begin{proof}
We write $z\in\mathbb{B}^1$ as $z=re^{2\pi it},0\leq r<1,t\in T$. For a given $0\leq r<1,f(z)=f(re^{2\pi it})=\sum_{k=1}^{\infty}a_{\lambda_k}r^{\lambda_k}e^{2\pi \lambda_kit}$. Since $f$ is continuous with respect to $t\in T$,  hence it is $L^p$ integrable over $T$, $\forall p>0$. From Lemma \ref{lem:fourier} above, we know that the $L_p(T)$ norms of $f|_{\{|z|=r\}}$ are equivalent for all $p\geq1$. However, for any $q<1$, by H\"older's inequality we obtain

$$(\int_T|f|^q)^{\frac{1}{2}}(\int_T|f|^{\alpha})^{\frac{1}{2}}\geq \int|f|$$ where $\alpha=2-q>1$. Therefore, all the $L^p(T)$ norms of $f|_{\{|z|=r\}}$ are equivalent. This allows us to calculate the $L^p$ norm of $f$ using its $L^2$ norm as follows: $$\int_{B(0,1)}|f|^p=\int_0^1||f|_{\{|z|=r\}}||_p^pdr\approx\int_0^1||f|_{\{|z|=r\}}||_2^pdr$$ $$=\int_0^1(\sum_{k=1}^{\infty}|a_{\lambda_k}|^2r^{2\lambda_k})^{\frac{p}{2}}dr.$$ This completes the proof.

\end{proof}

Now we fix a lacunary sequence $\{\lambda_k\}$, saying $\{2^k\}$, and consider the subspace of $A^p(\mathbb{B}^1)$: $A^p_c(\mathbb{B}^1):=\{f\in A^p(\mathbb{B}^1):f(z)=\sum_{k=1}^{\infty}a_kz^{2^k}\}$. We can prove following 

\begin{thm}
$A^p_c(\mathbb{B}^1)$ is a closed subspace of $A^p(\mathbb{B}^1)$.
\end{thm}
\begin{proof} Let $A^p_c(\mathbb{B}^1)$ be any Cauchy sequence in $A^p(\mathbb{B}^1)$ with respect to the distance function of $A^p(\mathbb{B}^1)$, denoted by $\{f_n\}_{n=1}^{\infty}$, where $f_n(z)=\sum_{k=1}^{\infty}a_{n,k}z^{2^k}$. From the above theorem, it is easy to see that for every $k$, the sequence $\{a_{n,k}\}$ converges to a complex number $a_k$. We will now prove that $f(z):=\sum_{k=1}^{\infty}a_kz^{2^k}\in A^p(\mathbb{B}^1)$, which is the limit of the sequence $\{f_n\}$ in $A^p(B(0,1))$. \ From the above theorem, we have 
\begin{align*}
\int|f|^p&\approx\int_0^1(\sum_{k=1}^{\infty}|a_k|^2r^{2^{k+1}})^{\frac{p}{2}}dr\\
&=\lim_{N \to \infty}\int(\sum_{k=1}^N|a_k|^2r^{2^{k+1}})^{\frac{p}{2}}dr\\ 
&=\lim_{N \to \infty}[\lim_{n \to \infty}\int(\sum_{k=1}^N|a_{n,k}|^2r^{2^{k+1}})^{\frac{p}{2}}dr]\\ 
&\leq \lim_{n \to \infty}\int(\sum_{k=1}^{\infty}|a_{n,k}|^2r^{2^{k+1}})^{\frac{p}{2}}dr\\
&=\lim_{n \to \infty}\int|f_n|^p.
\end{align*} 
Therefore, $f\in A^p(\mathbb{B}^1)$.\ Next, we prove that in $A^p(\mathbb{B}^1)$,$\forall n$, 
\begin{align*}
\sum_{k=1}^Na_{n,k}z^{2^{k+1}}\rightrightarrows \sum_{k=1}^{\infty}a_{n,k}z^{2^{k+1}}=f_n(N\rightarrow\infty).
\end{align*}
It is sufficient to prove that for any $\varepsilon>0$, $\int(\sum_{k=N}^{\infty}|a_{n,k}|^2r^{2^{k+1}}dr)^{\frac{p}{2}}\textless \varepsilon,\forall \ n\geq N=N(\varepsilon)$. In fact, since $\{f_n\}$ is a Cauchy sequence, there exists $M_0$, such that when $n\geq M_0$, $$\int_0^1(\sum_{k=1}^{\infty}|a_{n,k}-a_{M_0,k}|^2r^{2^{k+1}})^{\frac{p}{2}}dr\textless \varepsilon$$ Also note that there exists $N_0$, such that $$\int_0^1(\sum_{k=N_0}^{\infty}|a_{M_0,k}|^2r^{2^{k+1}}|)^{\frac{p}{2}}dr\textless \varepsilon$$ Combining these two facts, we can conclude that when $n\geq M_0$, $$\int_0^1(\sum_{k=N_0}^{\infty}|a_{n,k}|^2r^{2^{k+1}})^{\frac{p}{2}}dr\leq C\varepsilon,$$ where $C$ is a positive constant only dependent on $p$, which implies uniform convergence.\ Finally, we prove that in $A^p(\mathbb{B}^1)$, $f_n\rightarrow f$.\ It suffices to prove that for any $\varepsilon>0$, there exists $N$, such that for $n>N$, $$\int_0^1(\sum_{k=N}^{\infty}|a_{n,k}-a_k|^2r^{2^{k+1}})^{\frac{p}{2}}dr\leq \varepsilon$$ 
We notice that
\begin{align*}
&\ \ \ \ \int_0^1(\sum_{k=N}^{\infty}|a_{n,k}-a_k|^2r^{2^{k+1}})^{\frac{p}{2}}dr\\
&\leq 2^{\frac{p}{2}}[\int_0^1(\sum_{k=N}^{\infty}|a_{n,k}|^2r^{2^{k+1}})^{\frac{p}{2}}dr+\int_0^1(\sum_{k=N}^{\infty}|a_k|^2r^{2^{k+1}})^{\frac{p}{2}}dr], \ \mathrm{if} \ p\leq2 \\
&\leq2^{p-1}[\int_0^1(\sum_{k=N}^{\infty}|a_{n,k}|^2r^{2^{k+1}})^{\frac{p}{2}}dr+\int_0^1(\sum_{k=N}^{\infty}|a_k|^2r^{2^{k+1}})^{\frac{p}{2}}dr], \ \mathrm{if} \ p\geq2.
\end{align*}
This together with the uniform convergence yields the desired result.
\end{proof}

\begin{rem}
Theorem \ref{thm: char} can also be used to give a similar characterization of a class of  holomorphic functions on the punctured disk to be $L^p$-integrable by considering the Laurent expansions.
\end{rem}

	\end{document}